\documentclass[preprint,12pt]{elsarticle}

\usepackage{natbib}
\usepackage{graphicx,amsmath,amsthm,amsfonts,color,amssymb,comment}
\usepackage{tikz}
\usepackage{subfiles} 
\usepackage[nameinlink]{cleveref} 
\usepackage{kpfonts}
\usepackage{etoolbox}
\usepackage{thmtools}
\usepackage{environ}

\def\bfm#1{\boldsymbol{#1}} 
\def\RR{\mathbb{R}}

\newtheorem{theorem}{Theorem}[section]
\newtheorem{lemma}[theorem]{Lemma}

\theoremstyle{definition}
\newtheorem{example}[theorem]{Example}
\journal{Computer Aided Geometric Design}

\usetikzlibrary{intersections}
\usetikzlibrary{arrows.meta}

\begin{document}

\begin{frontmatter}

\title{Curvature-Based Optimal Polynomial Geometric Interpolation of Circular Arcs}

\author[address1]{Ema \v{C}e\v{s}ek\corref{corauth}}
\ead{ema.cesek@fmf.uni-lj.si}
\cortext[corauth]{Corresponding author}

\author[address1,address2]{Ale\v{s} Vavpeti\v{c}}
\ead{ales.vavpetic@fmf.uni-lj.si}

\address[address1]{Faculty of Mathematics and Physics, University of Ljubljana, 
Jadranska 19, Ljubljana, Slovenia}
\address[address2]{Institute of Mathematics, Physics and Mechanics, Jadranska 19, Ljubljana, Slovenia}

\begin{abstract}
The problem of the optimal approximation of circular arcs by parametric polynomial curves is considered. The optimality relates to the curvature error. Parametric polynomial curves of low degree are used and a geometric continuity is prescribed at the boundary points of the circular arc. Analysis is done for cases of parabolic $G^0$, cubic $G^1$ and quartic $G^2$ interpolation. The comparison of the approximation of circular arcs based on curvature with the approximation based on radial error is provided.
\end{abstract}


\end{frontmatter}


\section{Introduction}\label{sec:introduction}
The approximation of circular arcs is not only interesting from a theoretical perspective but also important for practical applications. Many control systems use parametric polynomials as basic building blocks. Since circular arcs cannot be represented in this way, it is important to find the best possible polynomial approximations. Several error functions can be used to measure the quality of approximants. If we are only interested in the visual image of the approximant, it is probably best to minimise the Hausdorff distance between the circular arc and the approximant. It is well known that the Hausdorff distance between a circular arc and a curve is (under some conditions) equal to the maximum of the radial distance between the arc and the approximant \cite{Jaklic-Kozak-2018-best-circle}. This fact makes the analysis of the problem easier. Still, the error function is irrational since part of it is the root of a polynomial function. In the first attempts to find the best approximation, the radial distance was replaced by the corrected polynomial error. The quadratic $G^0$ case was studied by M{\o}rken \cite{Morken-93-Parametric-Parabola}, and the cubic $G^1$ and the quartic $G^2$ cases were considered by Hur and Kim \cite{HurKim-2011}. Later, the problem was also solved with respect to the radial distance and for every inner angle $2\varphi\in (0,\pi]$ of the circular arc (see \cite{Vavpetic-Zagar-21-circle-arcs-Hausdorff}) and for some other cases with more parameters (\cite{Vavpetic-Zagar-2018-optimal-circle-arcs}, \cite{Vavpetic-2020-optimal-circular-arcs}).

Instead of the radial distance, the error function can be used to that measure how the curvature of the approximant matches the curvature of the circular arc. Such an approximation is useful, for example, in cases where the curve is considered as the trajectory of a moving body. The curvature function of the curve itself is more complicated than the distance function, so the analysis of finding the optimal approximant is more difficult and appears less frequently in the literature. Since the curvature of a circular arc is constant, one can determine the minimum value of the integral of the absolute value of the derivative of its curvature to measure the change of curvature of the approximant (see \cite{Jaklic-Zagar-11}). However, more often the error function is taken to be the maximum of the absolute value of the difference between the curvature of the circular arc and of the approximant. Dokken, D{\ae}hlen, Lyche and M{\o}rken \cite{Dokken-Daehlen-Lyche-Morken-90-CAGD} constructed a cubic $G^1$ approximant $\bfm{p}$ of a circular arc with the property that $\|\bfm{p}(\cdot)\|^2-1$ is an equioscillating function; i.e.\ all five local extrema have the same absolute value. In \cite{deBoor-Hoellig-Sabin-87-High-Accuracy} de Boor, H\"olling and Sabin described a parametric cubic spline interpolation scheme for plane curves which works very well for circular arcs. In both papers, the authors also realised that the curvature of the approximant of the circular arc is very close to 1, the curvature of the unit circular arc. The graph of the curvature function of the approximant oscillates five times around the constant 1, which is the same property as for the graph of the radial distance of the approximant. The difference is that the local extrema agree in absolute value in the case of radial distance, whereas in the curvature, they do not. For an optimal approximant, it might be natural to require that the curvature is an equioscillating function. Unfortunately, it turns out that there are no such approximants except for only one inner angle $2\varphi$. Kova\v c and \v Zagar \cite{Kovac-Zagar-curvature16} impose the condition that the curvature of the approximant coincides with the circular arc's curvature at the midpoint and at the edges. They proved that their approximants have an optimal asymptotic approximation order. We show that their construction provides the optimal cubic $G^1$ approximant only for inner angles smaller than some angle, which is calculated in the paper and at which the curvature is an equioscillating function, and that there are better approximations for circular arcs with an inner angle larger than the above one. We also show that something similar holds for the case of the quadratic $G^0$ and the quartic $G^2$ approximants. 

The paper is organized as follows. After the introduction in \Cref{sec:introduction}, some basic preliminaries are explained in \Cref{sec:preliminaries}. In \Cref{sec:G0_quadratic} the quartic $G^0$ case is considered and the main theorem \Cref{mainG0} describes the optimal approximants. In \Cref{sec:G1_cubic} the cubic $G^1$ case is considered. In the beginning, two cases motivate the analysis of the error function, and the main theorem \Cref{mainG1}, which gives the algorithm for finding the optimal approximant, is stated at the end of the section. \Cref{sec:G2_quartic} deals with the quartic $G^2$ case, where the algorithm for finding the optimal approximant is stated in \Cref{mainG2}. Numerical examples are shown at the end of \Cref{sec:G1_cubic} and \Cref{sec:G2_quartic}. Some closing remarks are given in the last section.

\section{Preliminaries}\label{sec:preliminaries}

Let us denote the considered circular arc with the inner angle $2\varphi$ by $\bfm{c}$. More precisely, $\bfm{c}$
will be parametrized as $\bfm{c}\colon[-\varphi,\varphi]\to\RR^2$, $0<\varphi\leq  \tfrac\pi 2$. It is enough to consider the unit circular arc centered at the origin of a particular coordinate system and symmetric with respect to the first coordinate axis. We can assume that $\bfm{c}(\alpha)=(\cos\alpha,\sin\alpha)^T$. Let the polynomial approximation of degree $n$ of $\bfm{c}$ be denoted by $\bfm{p}_n\colon[-1,1]\to\RR^2$, where $\bfm{p}_n=(x_n,y_n)^T$, and $x_n$, $y_n$ are scalar polynomials of degree at most $n$. The Bernstein-B\'ezier representation of $\bfm{p}_n$ will be considered, i.e.\
\begin{equation*}\label{p_Bern_Bez_form}
  \bfm{p}_n(t)=\sum_{j=0}^n B_j^n(t)\,\bfm{b}_j,
\end{equation*}
where $B_j^n$, $j=0,1,\dots,n$, are (reparameterized) Bernstein polynomials over $[-1,1]$, given as
\begin{equation*}
  B_j^n(t)=\binom{n}{j}\left(\frac{1+t}{2}\right)^{j}\left(\frac{1-t}{2}\right)^{n-j},
\end{equation*}
and $\bfm{b}_j\in\RR^2$, $j=0,1,\dots,n$, are the control points. Note that since $\bfm{c}$ is symmetric with respect to the first coordinate axis, the best parametric polynomial interpolant must possess the same symmetry, i.e.\ $\bfm{b}_j=r(\bfm{b}_{n-j})$, $j=0,1,\dots,n$, where $r\colon\RR^2\to\RR^2$ is the reflection over the first coordinate axis.

Circular arc $\bfm{c}$ and the parametric polynomial $\bfm{p}_n$ have the geometric contact of order $k$ at the boundary points $\bfm{c}(\pm \varphi)$, if there exists a smooth regular reparametrisation $\rho\colon [-1,1]\to[-\varphi,\varphi]$ with $\rho'(\pm 1)>0$, such that
\begin{equation*}
    \frac{d^j\bfm{p}_n}{dt^j}(\pm 1)=\frac{d^j(\bfm{c}\circ\rho)}{dt^j}(\pm 1),\quad j=0,1,\ldots,k.
\end{equation*}
In such case, we say that $\bfm{p}_n$ is a $G^k$ approximation of $\bfm{c}$. In particular, the $G^0$ approximation is equivalent to the interpolation of boundary points, the $G^1$ approximation provides that $\bfm{c}$ and $\bfm{p}_n$ share the same tangents at the boundary points and the $G^2$ approximation additionally requires that $\bfm{c}$ and $\bfm{p}_n$ have the same curvature at the boundary points.

Since the curvature of $\bfm{c}$ is identically 1, the appropriate signed curvature error function for $\bfm{p}_n$ is
$$
e_n(t)=1-\frac{\bfm{p}'_n(t)\times \bfm{p}''_n(t)}{\|\bfm{p}'_n(t)\|^3}.
$$
The goal of the approximation is thus to find a parametric polynomial $\bfm{p}_n$ satisfying the required geometric continuity for which the maximum of the absolute value of the function $e_n$ is as small as possible. 
The cubic $G^1$ and the quartic $G^2$ cases are considered in this paper. In both cases, the set of all polynomial approximants is a one-parametric family. Since a free parameter will be related to distance, we later denote it by $d$, and the function $e_n$ becomes a function of parameters $t$ and $d$.

Throughout the paper, we will write some polynomials in the basis $\mathcal{B}_x^n=\{x^i(1-x)^{n-i}\}_{0\le i\le n}$ instead of the standard basis $\mathcal{S}_x^n=\{x^i\}_{0\le i\le n}$. We have
\begin{align*}
p(x)&=\sum_{i=0}^n a_i x^i=\sum_{i=0}^n a_i x^i(x+1-x)^{n-i}=
\sum_{i=0}^n \sum_{j=0}^{n-i} \binom{n-i} j a_i x^{i+j}(1-x)^{n-i-j}\\
&=\sum_{i=0}^n \sum_{j=i}^n \binom{n-i} {j-i} a_i x^j(1-x)^{n-j}
=\sum_{j=0}^n \left(\sum_{i=0}^j \binom{n-i} {j-i} a_i\right) x^j(1-x)^{n-j},
\end{align*}
so the transition matrix is lower triangular, with binomial coefficients for the elements. 
By analogy, the polynomial of two variables can be written in $\mathcal{B}_x^n\times \mathcal{B}_y^m$ basis instead of $\mathcal{S}_x^n\times \mathcal{S}_y^m$ basis, and we have
\begin{align*}
p(x,y)&=\sum_{i=0}^n\sum_{j=0}^m a_{i,j} x^iy^j=
\sum_{k=0}^n\sum_{l=0}^m\left(\sum_{i=0}^k\sum_{j=0}^l \binom{n-i}{k-i}\binom{m-j} {l-j} a_{i,j}\right) x^k(1-x)^{n-k}y^l(1-y)^{m-l}.
\end{align*}
The main reason for the new basis is the following argument, which will be used several times in the paper: if a polynomial $p$ in the basis $\mathcal{B}_x^n\times \mathcal{B}_y^m$ has all coefficients nonnegative, then $p(x,y)\ge0$ provided $0\le x,y\le 1$. We can do the same for polynomials with several variables.

\section{Quadratic $G^0$ approximants}\label{sec:G0_quadratic}

This case may not be interesting in practice, because for trajectory approximations we often prescribe not only the edge points but also the start and the end directions, so we are looking for at least a $G^1$ approximant. Although, the analysis of this case is not difficult we belive that the result is surprising.
Geometrically, we are looking for a parabola that passes through given boundary points of the circular arc with the curvature as close as possible to the curvature of the arc.
If for a fixed $\varphi \in (0,\frac{\pi}{2}]$ we denote $c=\cos\varphi\in [0,1)$, the control points are
$$
\bfm{b}_0=(c,-\sqrt{1-c^2})^T,\bfm{b}_1=(d,0)^T,
  \bfm{b}_2=(c,\sqrt{1-c^2})^T,
$$
where $d>0$. The corresponding signed curvature error function is
$$
e_2(t,d)=1+\frac{\sqrt{1-c^2} (c-d)}{\sqrt{\left(1-c^2+(c-d)^2 t^2\right)^3}}.
$$
Since $\tfrac{de_2}{dt}(t,d)=-3 \sqrt{1-c^2} (c-d)^3 t\left(1-c^2+(c-d)^2 t^2\right)^{-5/2}$, the only candidates for the extrema of $e_2(\cdot,d)$ on [-1,1] are $0$ and $\pm 1$.

If $d<c$, the graph of the function $e_2(\cdot,d)$ is above the line $c=1$, so the optimal parameter must be greater than $c$; in fact the graph of the function $e_2(\cdot,d)$ is the mirror image over horizontal line $c=1$ of the graph of the function $e_2(\cdot,2c-d)$. And for $d>c$ we have $\tfrac{d^2e_2}{dt^2}(0,d)=\tfrac{3(d-c)^3}{(1-c^2)^2}>0$, so $e_2(\cdot,d)$ has minimum at $t=0$ and maximum at $t=\pm 1$.

The only value $d>c$ for which $\tfrac{de_2}{dd}(1,d)=0$ is $d_e(c):=c+\tfrac{\sqrt{2}}2\sqrt{1-c^2}$.
If $|e_2(0,d_e(c))|\le |e_2(1,d_e(c))|$, then $d_e(c)$ is the optimal parameter. As the examples below show, both possibilities exist.

\begin{example}
Let $\varphi_1=\tfrac{\pi}4$, then $c_1=\tfrac{\sqrt 2}2$ and $d_e(c_1)=\tfrac{1+\sqrt 2}2$. Since $e_2(0,d_e(c_1))=0$, $d_e(c_1)$ is the optimal parameter (see the left graph on \Cref{fig:e2examples}).

\begin{figure}[h!]
\begin{minipage}{.5\textwidth}
  \centering
 \includegraphics[width=0.8\textwidth]{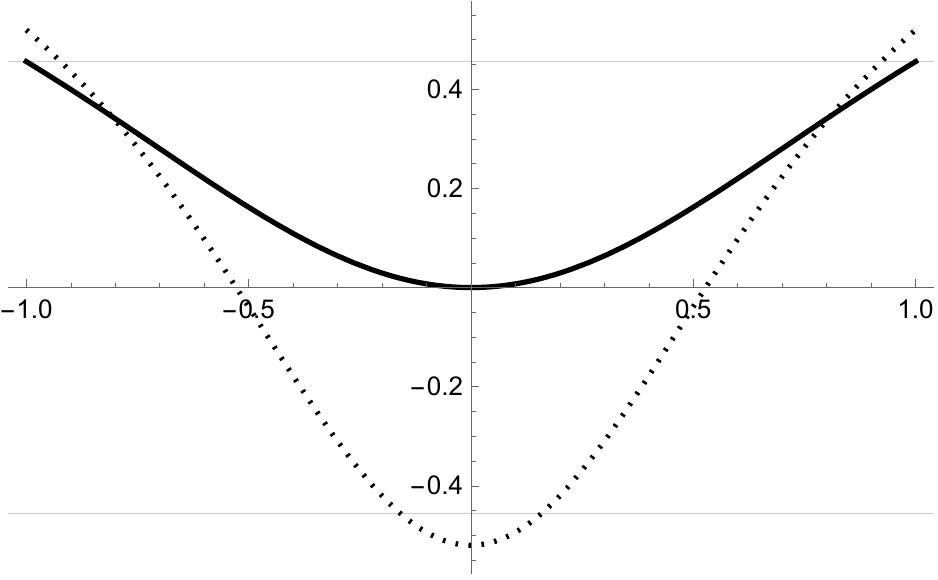}
\end{minipage}
\begin{minipage}{.5\textwidth}
\centering
  \includegraphics[width=0.8\textwidth]{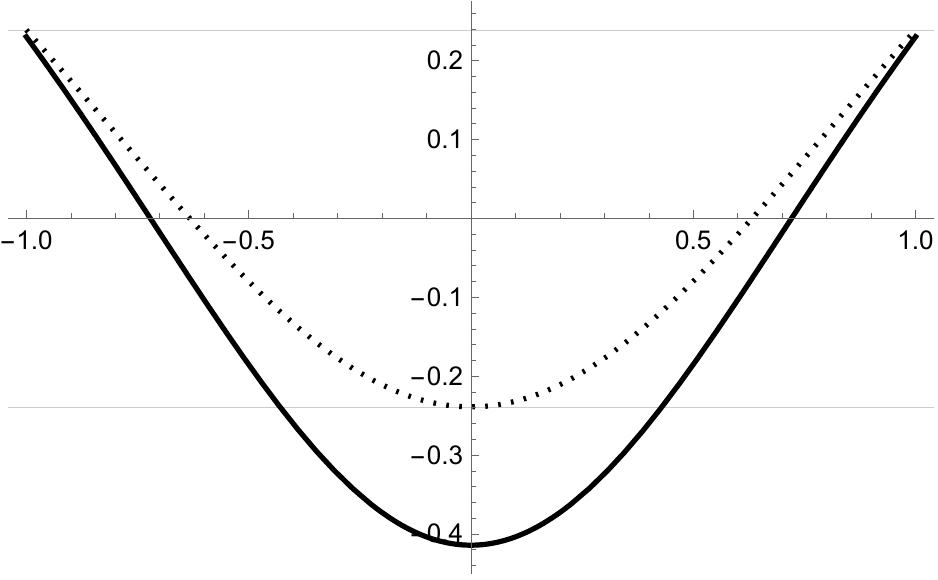}
\end{minipage}
  \caption[]{On the left it is the graph of the error function of the optimal approximant for $c=\tfrac{\sqrt 2} 2$; the solid one is the graph of the error function of the optimal approximant and the dashed one of the approximant that has the property that the error function equioscillates. On the right it is the graph of the error function of the optimal approximant for $c=\frac {\sqrt{3} }{2}$; the solid one is the graph of the error function of the approximant induced by $d_e(c_2)$  and the dashed one of the optimal approximant, for which, in this case, the error function equioscillates.}
  \label{fig:e2examples}
\end{figure}

Let $\varphi_2=\tfrac{\pi}6$, then $c_2=\frac{\sqrt{3} }{2}$ and $d_e(c_2)=\tfrac{\sqrt 2+2\sqrt 3}4$. Since $|e_2(0,d_e(c_2))|=\sqrt 2-1>\tfrac{9-4\sqrt 3}9=e_2(1,d_e(c_2))$, $d_e(c_2)$ is not the optimal parameter. There exists $d<d_e(c_2)$ for which $|e_2(0,d)|=e_2(1,d)$ and induces a better approximant than $d_e(c_2)$ as it is show on the right graph on \Cref{fig:e2examples}.
\end{example}

Let us denote the optimal parameter $d$ by $d^*(c)$.
\begin{theorem}\label{mainG0}
Let $c\in [0,1)$.
\begin{enumerate}
\item If $c\le \tfrac{\sqrt 6}{36} \sqrt{181-12 \sqrt{6}}\approx 0{.}83778$, then $d^*(c)=d_e(c)=c+\tfrac{\sqrt{2}}2\sqrt{1-c^2}$.
\item If $c> \tfrac{\sqrt 6}{36} \sqrt{181-12 \sqrt{6}}$ then $d^*(c)$ is the unique solution of $e_2(1,d)=-e_2(0,d)$ on $[c,d_e(c)]$.
\end{enumerate}
\end{theorem}

\begin{proof}
\begin{enumerate}
\item If $e_2(1,d_e(c))+e_2(0,d_e(c))\ge 0$, then $d^*(c)=d_e(c)$. The statement follows since the inequality holds for $c\le \tfrac{\sqrt 6}{36} \sqrt{181-12 \sqrt{6}}$.
\item If $c> \tfrac{\sqrt 6}{36} \sqrt{181-12 \sqrt{6}}$ then $e_2(1,d_e(c))+e_2(0,d_e(c))< 0$. For every $d\in (c,d_e(c))$ we have $\tfrac{de_2}{dd}(1,d)=-\left(1-3 c^2+4 c d-2 d^2\right) \sqrt{1-c^2}\left(1-2 c d+d^2\right)^{-5/2}<0$, hence the function $e_2(1,\cdot)$ is strictly decreasing. The function $e_2(0,\cdot)$ is strictly decreasing since $\tfrac{de_2}{dd}(0,d)=-\tfrac1{1-c^2}<0$. Therefore, $e_2(0,\cdot)+e_2(1,\cdot)$ is strictly decreasing on $[c,d_e(c)]$. Since $e_2(0,c)+e_2(1,c)=2>0$ and $e_2(1,d_e(c))+e_2(0,d_e(c))< 0$ there exists the unique parameter $d^*(c)\in [c, d_e(c)]$ such that $e_2(0,d^*(c))+e_2(1,d^*(c))=0$. For $d<d^*(c)$ we have $e_2(1,d)>e_2(1,d^*(c))=\max_{t\in[-1,1]} e_2(t,d^*(c))$ and for $d>d^*(c)$ we have $|e_2(0,d)|>|e_2(0,d^*(c))|=\max_{t\in[-1,1]} e_2(t,d^*(c))$, so $d^*(c)$ is the optimal parameter.\qedhere
\end{enumerate}
\end{proof}
We have seen that although there are approximants for which the error function is equioscillating, in some cases, the error function of the optimal approximant does not have this property. Moreover, the error function can be nonnegative, as shown by the example $\varphi=\tfrac\pi 4$ (see \Cref{fig:e2examples}).

\section{Cubic $G^1$ approximants}\label{sec:G1_cubic}

In this section, we consider the $G^1$ approximation of the circular arc $\bfm{c}$, where $\varphi\in (0,\tfrac\pi 2]$ is fixed. If we again denote $c=\cos\varphi\in [0,1)$, the control points are
\begin{align*}
  \bfm{b}_0&=(c,-\sqrt{1-c^2})^T,
  &\bfm{b}_1=(c,-\sqrt{1-c^2})^T+d\,(1-c^2,c\sqrt{1-c^2})^T,\\
  \bfm{b}_3&=(c,\sqrt{1-c^2})^T,
  &\bfm{b}_2=(c,\sqrt{1-c^2})^T+d\,(1-c^2,-c\sqrt{1-c^2})^T,
\end{align*}
where $d>0$.

The corresponding signed curvature error function is
$$
e_3(t,d)=1+\frac{8 d \left(c d-2+(3 c d-2) t^2\right)}{3 \sqrt{\left((2-c d)^2+2 (d (2 d+c (8-5 c d))-4) t^2+(2-3 c d)^2 t^4\right)^3}}=:1+\frac{f(t,d)}{\sqrt{g(t,d)^3}}.
$$
Since $e_3(0,d)=1-\frac{8 d(2-cd)}{3 |2-c d|^3}$ and $e_3(1,d)=1-\frac{4 (1-c d)}{3 d^2}$, the function $e_3(0,\cdot)$ is strictly decreasing and the function $e_3(1,\cdot)$ is strictly increasing for $0<d<\tfrac 2 c$ if $c \in (0,1)$ or for $d>0$ if $c=0$. For $d>\tfrac 2 c$, the curvature at $t=0$ is greater than one, which means that such $d$ does not induce an optimal approximant. Hence we may assume that $d<\tfrac 2 c$ if $c \ne 0$ and in this case $e_3(1,\cdot)$ and $e_3(0,\cdot)$ are strictly monotone rational functions.

The equation $e_3(0,d) = e_3(1,d)$ implies the unique solution
$$
d_e(c) :=\frac{5 c^2+\left(5 \sqrt{3}-\sqrt{27+c^3}\right) \sqrt[3]{\sqrt{27+c^3}+3 \sqrt{3}}-\left(5 \sqrt{3}+\sqrt{27+c^3}\right) \sqrt[3]{\sqrt{27+c^3}-3 \sqrt{3}}}{3 \left(2+c^3\right)}.
$$
In the examples below we show that $e_3(0,d_e(c))$ is equal $\max |e_3(\cdot,d_e(c))|$ for some values $c$ but not for all.
If $e_3(0,d_e(c))=\max |e_3(\cdot,d_e(c))|$, then $d_e(c)=d^*(c)$ and it induces the optimal approximation, since for $d>d_e(c)$ we have $e_3(1,d)>e_3(1,d_e(c))=\max |e_3(\cdot,d_e(c))|$ and for $d<d_e(c)$ we have $e_3(0,d)>e_3(0,d_e(c))=\max |e_3(\cdot,d_e(c))|$. 

\begin{example}\label{example0.5}
Let $c=\tfrac 1 2$, i.e.\ the inner angle of the circular arc is $\tfrac{2\pi}3$. Then 
\begin{align*}
d_e\left(\frac{1}{2}\right)&=\frac{2}{51} \left(5-\frac{383}{\sqrt[3]{12671+408\sqrt{1302}}}+\sqrt[3]{12671+408\sqrt{1302}}\right)\approx 0.8800,\\
e_3\left(0,d_e\left(\frac{1}{2}\right)\right)&=\max \left|e_3\left(\cdot,d_e\left(\frac{1}{2}\right)\right)\right| \\
&=\frac{1}{36} \left(38-\frac{193}{\sqrt[3]{14113-384\sqrt{1302}}}-\sqrt[3]{14113-384 \sqrt{1302}}\right)\approx 0.0358.
\end{align*}
In this case, $d_e\left(\frac{1}{2}\right)$ induces the optimal approximant (see the left graph on \Cref{fig:e3examples}). 
\end{example}

\begin{figure}[h!]
\begin{minipage}{.5\textwidth}
  \centering
 \includegraphics[width=0.9\textwidth]{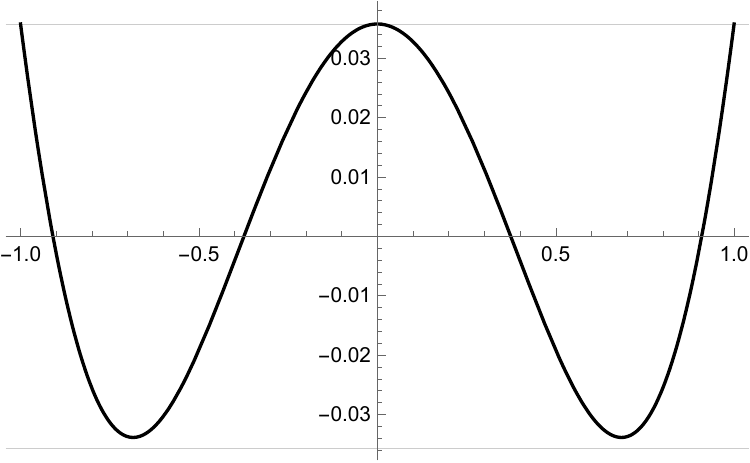}
\end{minipage}
\begin{minipage}{.5\textwidth}
\centering
  \includegraphics[width=0.9\textwidth]{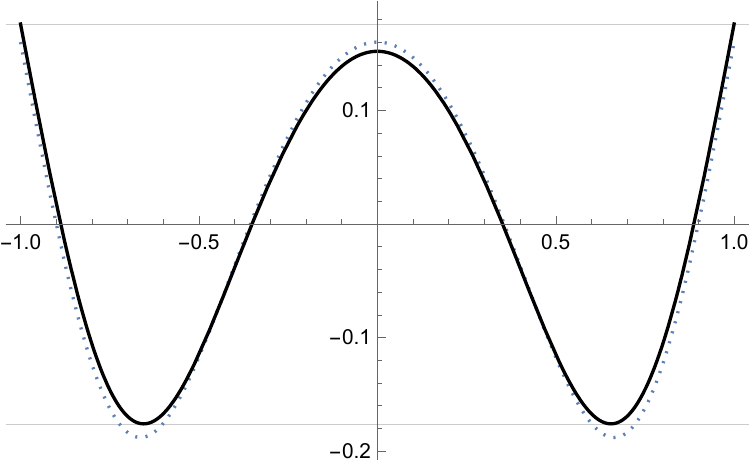}
\end{minipage}
  \caption[]{On the left it is the graph of the error function of the optimal approximant for $c=\tfrac 1 2$. On the right, the solid graph is the graph of the error function of the optimal approximant for $c=0$, and the dotted graph is the graph of the error function where the values at $t=0$ and $t=1$ coincide.}
  \label{fig:e3examples}
\end{figure}

\begin{example}\label{example0}
Let $c=0$, i.e.\ we consider the half circle. Then $d_e(0)=\sqrt[3]{2}\approx 1.2599$. But $e_e(0,d_e(0))=1-\frac{2 \sqrt[3]{2}}{3}\approx0.1601<0.1880\approx |\min e_3(\cdot, d_e(0))|$. Thus, $d_e(0)$ almost certainly does not induce the optimal approximant. It is easy to see that the optimal parameter $d^*(0)$ is on the interval $[\tfrac{2\sqrt{3}}{3},\tfrac{3}{2}]$, since $e_3(1,\tfrac{2\sqrt{3}}{3})=0$ and $e_3(0,\tfrac{3}{2})=0$. For every $d\in [\tfrac{2\sqrt{3}}{3},\tfrac{3}{2}]$ we have 
\begin{align*}
\min e_3(\cdot,d)=1-\frac{32 \sqrt{2} d \left(4-d^2+\sqrt{144-40 d^2+d^4}\right)}{9\sqrt{3 \left(4-d^2\right)^3 
 \left(12+d^2-\sqrt{144-40 d^2+d^4}\right)^3}}
\end{align*}
which is an increasing function on $(0,\tfrac 3 2)$, since it is easy to observe that its derivative is positive. Hence there exists a unique $d^*(0)\in [\tfrac{2\sqrt{3}}{3},\tfrac{3}{2}]$ such that $e_3(1,d^*(0))=-\min e_3(\cdot,d^*(0))$. More precisely, $d^*(0)\approx 1.2721$ induces the optimal approximant (see the right graph on \Cref{fig:e3examples}).
\end{example}

\begin{figure}[h!]
\begin{minipage}{.5\textwidth}
  \centering
 \includegraphics[width=0.9\textwidth]{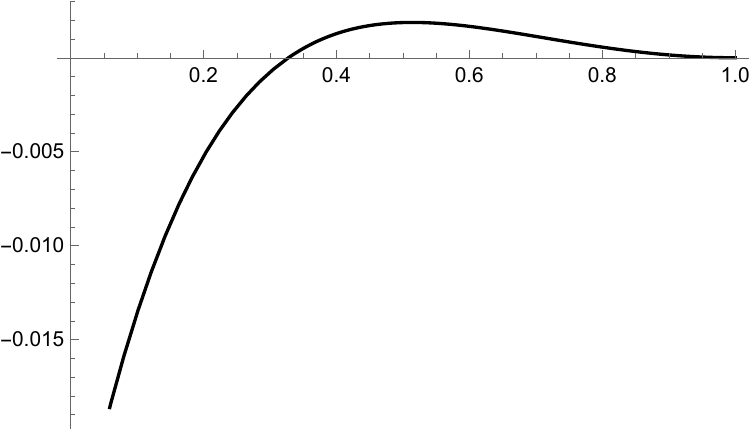}
\end{minipage}
\begin{minipage}{.5\textwidth}
\centering
  \includegraphics[width=0.9\textwidth]{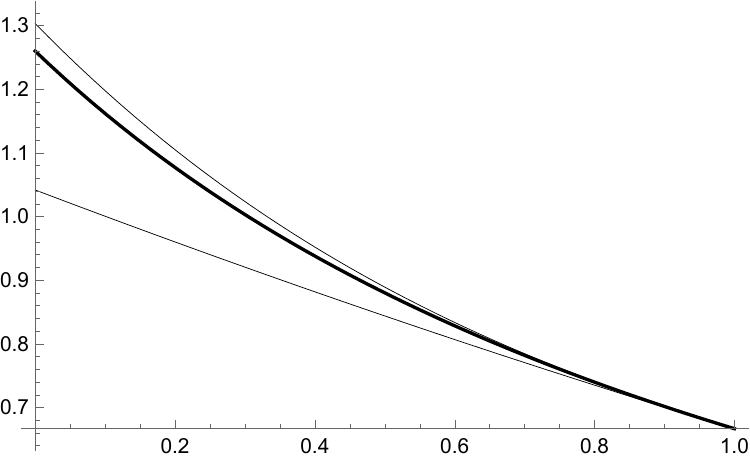}
\end{minipage}
  \caption[]{On the left it is the graph of the function $e_3(0,d_e(c))+\min e_3(\cdot,d_e(c))=\max e_3(\cdot,d_e(c))+\min e_3(\cdot,d_e(c))$. On the right it is a graph of function $d_e$ with its upper and lower bounds $d_1$ and $d_2$.}
  \label{fig:e3bounds}
\end{figure}

The last example shows that for some values of $c$, it is necessary to analyze the behaviour of the minimum of the function $e_3(\cdot,d)$. In principle, it is possible to explicitly calculate the function $e_3(0,d_e(c))+\min e_3(\cdot,d_e(c))$, but its expression is very long. From its graph (see left graph on \Cref{fig:e3bounds}), it seems that only for $c\ge 0.32792$, i.e.\ for a circular arc with the inner angle $2\varphi\le 0.78731\pi$, the method $3$ from \cite{Kovac-Zagar-curvature16} produces the optimal approximant. However, the function is too complicated to determine its zero analytically.

Since it is difficult to determine the minimum value of the error function as it depends on $c$ and $d$, for each $c$ we will find (as small as possible) interval $I_c$ such that the optimal parameter $d^*(c)\in I_c$. For all parameters $d\in I_c$ we will determine the necessary properties of the function $e_3(\cdot,d)$ to be able to determine the optimal parameter. Since we want the interval $I_c$ to be small, let's look at the power series expansion for $d_e$ around the point $c = 1$ and let us define
\begin{align*}
d_1(c)&:=\frac{2}{3}+\frac{1}{3}(1-c)+\frac{1}{24} (1-c)^2,\\
d_2(c)&:=\frac{2}{3}+\frac{1}{3}(1-c)+\frac{1}{6} (1-c)^2+\frac{101}{1152}(1-c)^3+\frac{25}{512} (1-c)^4.
\end{align*}
\begin{figure}[ht]
\begin{minipage}{.5\textwidth}
  \centering
 \includegraphics[width=0.9\textwidth]{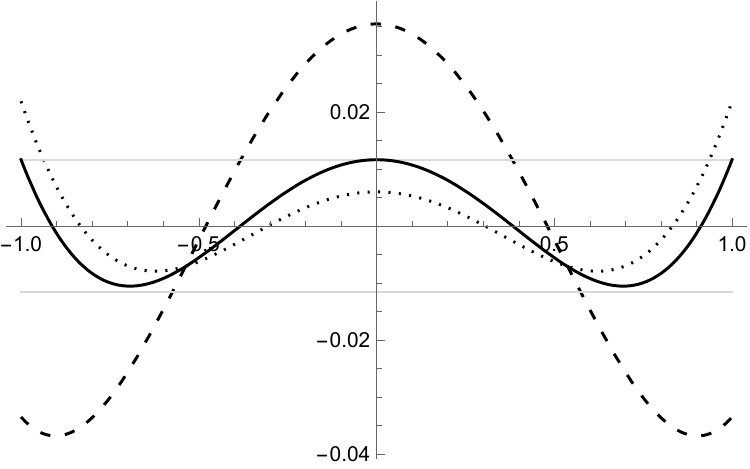}
\end{minipage}
\begin{minipage}{.5\textwidth}
\centering
  \includegraphics[width=0.9\textwidth]{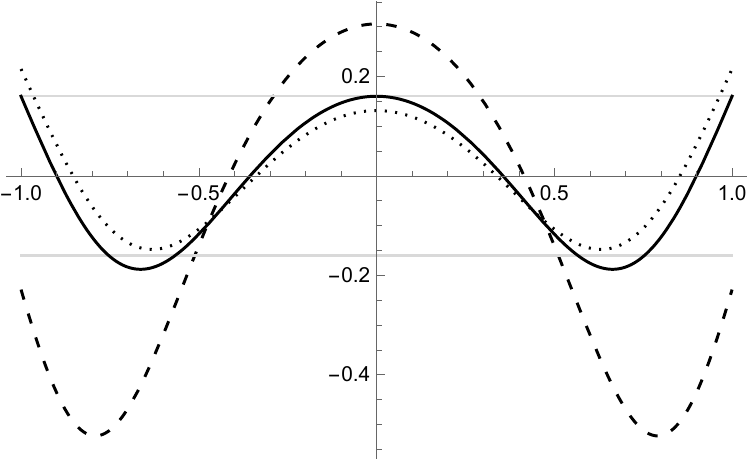}
\end{minipage}
  \caption[]{Graphs of $e_3 (\cdot,d_e(c))$ (solid), $e_3(\cdot,d_1(c))$ (dashed) and $e_3(\cdot,d_2(c))$ (dotted) for $\varphi = \tfrac \pi 4$ and $\varphi = \tfrac \pi 2$.}

  \label{fig:e3kot}
\end{figure}

The graphs of functions $d_e$, $d_1$ and $d_2$ are shown in the right graph of \Cref{fig:e3bounds} and two examples of the graphs of the curvature function of parametric polynomial curves induced by $d_e(c)$, $d_1(c)$ and $d_2(c)$ are on \Cref{fig:e3kot}.
Before we show that $d^*(c)\in [d_1(c),d_2(c)]=I_c$ (\Cref{bounds_for_de}), we need the following technical lemma. The main point of the lemma is to show that for every $d\in I_c$, the graph of the function $e_3(\cdot,d)$ looks like the graph of an even polynomial of degree four with a positive leading coefficient (see \Cref{fig:e3kot}), and that the minimum of the function $e_3(\cdot,d)$ is a strictly increasing function of the parameter $d\in I_c$. Note that the last property is the one we needed in \Cref{example0}.

\begin{lemma} \label{properties_of_e3}
Let $c\in (0,1)$.
\begin{enumerate}
\item We have $d_2(c) < \tfrac{2}{3c}$. 
\item If $d\in(0,\tfrac{2}{3c})$ the function $e_3(\cdot,d)$ has at most three local extrema.
\item For all $d\in I_c$, the function $e_3(\cdot,d)$ has the local maximum at $t=0$ and the minimum at a point on the interval $(\tfrac 3 5,1)$.
\item For $t\in [\tfrac 3 5,1)$ the function $e_3(t,\cdot)$ is strictly increasing on $I_c$.
\end{enumerate}
\end{lemma}

\begin{proof}
\begin{enumerate}
\item The function $h$ defined by $h(c) = \tfrac{2}{3c}- d_2(c)$ is strictly decreasing since all coefficients of $c^2h'(c)$ are negative in the basis $\mathcal{B}_c^5$. Since $h(1) = 0$, the function $h$ is positive on $(0,1)$ and therefore $d_2(c) < \tfrac{2}{3c}$ provided $c\in (0,1)$.

\item The equation $\frac{de_3}{d t}(t,d)=0$ has at most five real solutions. One of them is $t=0$. We only need to prove that the equation has at most one solution on the interval $(0,1]$, due to $e_3(t,d)=e_3(-t,d)$.
The only candidates for positive solutions of the equation are
\begin{align*}
t_{1,2} & = \frac{\sqrt{-h(d)\pm\sqrt{112 c^4 d^4-512 c^3 d^3-32 c^2 d^4+832 c^2 d^2+80 c d^3-576 c d+d^4-40 d^2+144}}}{\sqrt{2}(2-3 c d)}, 
\end{align*}
where 
$$
h(d) = 4-8 c d+d^2+2 c^2 d^2=\left(2 c^2+1\right) \left(d-\frac{4 c}{2 c^2+1}\right)^2+\frac{4-8 c^2}{2 c^2+1}.
$$
It suffices to show that $h(d)>0$ for $d\in(0,\tfrac{2}{3c})$, because then the solution $t_2$ cannot be real. If $c\in (0, \tfrac{1}{2})$, then $h$ has the minimum at $d=\frac{4 c}{2 c^2+1}$ and $h\left(\frac{4 c}{2 c^2+1} \right)=\frac{4-8 c^2}{2 c^2+1}>0$. If $c\in [\tfrac{1}{2}, 1)$ then $h$ is strictly decreasing, since $h'(d)=-\left(2+4 c^2\right) \left(\frac{2}{3 c}-d\right)-\frac{4}{3 c} \left(4 c^2-1\right)<0$, and $h(\tfrac{2}{3c})=\tfrac{4}{9c^2}(1-c^2) \ge 0$. Hence, $h(d)>0$, so the solution $t_1$ is the only candidate for a local extremum of the function $e_3(\cdot,d)$ on the interval $(0,\infty)$, which proves the statement.

\item By the previous statement, it is enough to prove that $\frac{de_3}{d t}\left(\tfrac 3 5,d\right)<0$ and $\frac{de_3}{d t}(1,d)>0$ for all $d\in I_c$. 
We have
\begin{align*}
\frac{de_3}{d t} \left(\tfrac 3 5,d\right) =\frac{-6250 d \cdot h_1(d)}{\left(-224 c^2 d^2+32 c d+225 d^2+256\right)^{5/2}} \, \,\text{ and } \, \,
\frac{de_3}{d t}(1,d) =\frac{h_2(d)}{3d^3},
\end{align*}
where 
$h_1(d) =  -1044 c^3 d^3+5288 c^2 d^2+1275 c d^3-7892 c d-2100 d^2+3776$ and
$h_2(d)= 8d - 24c + 48 c^2 d - 6 c d^2 - 24 c^3 d^2$.
Since $h_1((1-\delta)\cdot d_1(c)+ \delta \cdot d_2(c))$ and $h_2((1-\delta)\cdot d_1(c)+ \delta \cdot d_2(c))$ have all coefficients positive in bases $\mathcal{B}_c^{15}\times \mathcal{B}_\delta^3$ and $\mathcal{B}_c^{11}\times \mathcal{B}_\delta^2$, respectively, we conclude that $\frac{de_3}{dt}\left(\tfrac 3 5,d \right) < 0$  and $\frac{de_3}{dt}(1,d) > 0$ for $d\in I_c$. So the local maximum is at $t=0$ and the minimum of $e_3(\cdot,d)$ is on the interval $(\tfrac 3 5,1)$. 

\item Recall that $e_3(t,d)=1+\tfrac{f(t,d)}{\sqrt{g(t,d)^3}}$. Hence 
$$
\frac{de_3}{dd}(t,d)= \frac{\frac{df}{dd}(t,d)g(t,d)-\tfrac 3 2 f(t,d)\frac{dg}{dd}(t,d)}{\sqrt{g(t,d)^5}}=:\frac{h(t,d)}{\sqrt{g(t,d)^5}}.
$$
Since $h(\frac{3+2\tau}{5},(1-\delta)\cdot d_1(c)+ \delta \cdot d_2(c))$ has all coefficients positive in the basis $\mathcal{B}_c^{15}\times\mathcal{B}_\tau^{6}\times\mathcal{B}_\delta^{3}$, the derivative of $e_3(t,\cdot)$ is positive on $I_c$.
\qedhere\end{enumerate}
\end{proof}

Now we can prove that $d_1(c)$ and $d_2(c)$ are the bounds for the optimal parameter $d^*(c)$.

\begin{lemma} \label{bounds_for_de}
For every $c \in (0,1)$ we have $d_e(c),d^*(c)\in I_c$.
\end{lemma}

\begin{proof}
First, we show that $d_e(c)\in I_c$. Since $e_3(0,\cdot)$ is strictly decreasing, $e_3(1,\cdot)$ is strictly increasing and $e_3(0,d_e(c))=e_3(1,d_e(c))$ it is enough to show that $e_3(0,d_1(c))>e_3(1,d_1(c))$ and $e_3(0,d_2(c))<e_3(1,d_2(c))$ for all $c\in (0,1)$. Since
\[
\tfrac{(5-c)^4 \left(48-25 c+10c^2-c^3\right)^2}{32 (1-c)^2}\left(e_3(0,d_1(c))-e_3(1,d_1(c))\right)
\]
has all coefficients positive in the basis $\mathcal{B}_c^7$, we get that $d_1(c)$ is the lower bound for $d_e(c)$. And since
\[
\tfrac{\left(6005-5184 c+3330 c^2-1304 c^3+225 c^4\right)^2 \left(9216-6005 c+5184 c^2-3330 c^3+1304 c^4-225 c^5\right)^2}{6144 (1-c)^2}\left(e_3(1,d_2(c))-e_3(0,d_2(c))\right)
\]
has all coefficients positive in the basis $\mathcal{B}_c^{13}$, we get that $d_2(c)$ is the upper bound for $d_e(c)$.

Let us show that $d^*(c)\in I_c$. Above we have proved that $e_3(1,d_2(c))=\max e_3(\cdot,d_2(c))$. Let us show that $e_3(1,d_2(c))=\max |e_3(\cdot,d_2(c))|$. It suffices to prove that $e_3(1,d_2(c))+e_3(t,d_2(c))\ge 0$ for all $t\in [-1,1]$.
Recall that $e_3(t,d_2(c))=1+\tfrac{f(t,d_2(c))}{\sqrt{g(t,d_2(c))^3}}$. Since $e_3(1,d_2(c))>0$ and $f(t,d_2(c))<0$, we can square both sides of the inequality 
$e_3(1,d_2(c))+1 \ge - f(t,d_2(c))g(t,d_2(c))^{-3/2}$.
Hence, the inequality follows because 
\[ \left(6005-5184 c+3330 c^2-1304 c^3+225 c^4\right)^4 \left( (e_3(1,d_2(c))+1)^2 g(t,d_2(c))^3-f(t,d_2(c))^2 \right) \]
has all coefficients nonnegative in the basis $\mathcal{B}_c^{46}\times\mathcal{B}_{t}^{12}$. Since $e_3(1,\cdot)$ is strictly increasing, for $d>d_2(c)$, we have $\max |e_3(\cdot,d)|\ge e_3(1,d)>e_3(1,d_2(c))=\max e_3|(\cdot,d_2(c))|$, so $d_2(c)$ is the upper bound for $d^*(c)$.

Since 
$$
\tfrac{(5-c)^4 \left(6005-5184 c+3330 c^2-1304 c^3+225 c^4\right)^2}{2 (1-c)^2}(e_3(1,d_1(c))+e_3(1,d_2(c))),
$$
has all coefficients negative in the basis $\mathcal{B}_c^{10}$ and $e_3(1,\cdot)$ is strictly increasing, for $d<d_1(c)$ we have  $\max |e_3(\cdot,d)|\ge |e_3(1,d)|>|e_3(1,d_1(c))|>e_3(1,d_2(c))=\max|e_3(\cdot,d_2(c))|$. Therefore, $d^*(c)\ge d_1(c)$.
\end{proof}

By \Cref{properties_of_e3} for every $d\in I_c$ the function $e_3(\cdot,d)$ has the minimum at $t_1(d)\in (\tfrac 3 5,1)$ and the maximum at 0 or 1. The following theorem describes how to find the optimal parameter $d^*(c)$.

\begin{theorem}\label{mainG1}
Let $c\in [0,1)$.
\begin{enumerate}
\item If $e_3(1,d_e(c))\ge |e_3(t_1(d_e(c)),d_e(c))|$, then $d^*(c)=d_e(c)$.
\item If $e_3(1,d_e(c))< |e_3(t_1(d_e(c)),d_e(c))|$, then there exists a unique $d>d_e(c)$ such that $|e_3(t_1(d),d)|=e_3(1,d)$. Then $d^*(c)=d$.
\end{enumerate}
\end{theorem}

\begin{proof}
\begin{enumerate}
\item Since $e_3(1,\cdot)$ is increasing every $d>d_e(c)$ induces a worse approximant and since $e_3(0,\cdot)$ is decreasing every $d<d_e(c)$ induces a worse approximant, hence $d^*(c)=d_e(c)$.

\item  By \Cref{properties_of_e3} the function $e_3(t_1(\cdot),\cdot)$ is increasing, hence $|e_3(t_1(\cdot),\cdot)|$ is decreasing. Since $e_3(1,\cdot)$ is increasing there exists a unique $d\in (d_e(c),d_2(c)]$ such that $e_3(1,d)=|e_3(t_1(d),d)|$. For every $d'<d$ we have $|e_3(t_1(d'),d')|>|e_3(t_1(d),d)|$ and for every $d'>d$ we have $e_3(1,d')>e_3(1,d)>0$, hence $d=d^*(c)$. \qedhere
\end{enumerate}    
\end{proof}

Although the analysis of the error function itself is very technical, the algorithm for determining the optimal parameter is very simple. For some circular arcs, there are closed forms for the optimal parameter, as shown in Example 1. For some, the optimal parameter is obtained as the solution of a system of nonlinear equations. In the above analyses, we have found good bounds for the optimal parameter so we have good initial approximations for solving the system numerically.
In \Cref{table:G1}, numerical values of optimal parameters $d^*$ according to the curvature error and optimal parameters $d^r$ according to the radial error \cite{Vavpetic-Zagar-2018-optimal-circle-arcs} are shown for several values of an inner angle $2\varphi$.

Since $d^r > d^*$, the optimal approximant with respect to the radial error has maximal curvature at the boundary points. The graph of the curvature error function oscillates three times, but the maximum is always greater than the absolute value of the minimum (see the left graph on \Cref{fig:e3compare}). Also, the graph of the radial error of the optimal approximant with respect to the curvature error oscillates three times, and the maximal amplitude is at the midpoint (see the right graph on \Cref{fig:e3compare}).
\begin{table}[ht]
\begin{center}
\begin{tabular}{ |c|c|c|c|c|c|c| } 
 \hline
 $\varphi$ & $d^*$ & curvature error & radial error & $d^r$ & curvature error & radial error \\ 
   \hline
 $\frac{\pi}{2}$ & $1{.}272063$ & $1{.}76 \times 10^{-1}$ & $4{.}60 \times 10^{-2}$ & 
                    $1{.}315740$ & $2{.}30 \times 10^{-1}$& $1{.}32 \times 10^{-2}$\\ \hline
 $\frac{\pi}{3}$ & $0{.}879981$ & $3{.}58 \times 10^{-2}$ & $5{.}01\times 10^{-3}$ & 
                    $0{.}886910$ & $5{.}66\times 10^{-2}$ & $1{.}11\times 10^{-3}$\\ 
 \hline
 $\frac{\pi}{4}$ & $0{.}778639$ & $1{.}16 \times 10^{-2}$ & $9.03\times 10^{-4}$& 
                   $0{.}780526$ & $1.93\times 10^{-2}$& $1{.}96\times 10^{-4}$\\ 
 \hline
 $\frac{\pi}{6}$ & $0{.}714105$ & $2{.}33 \times 10^{-3}$ & $8{.}00\times 10^{-5}$ & 
                    $0{.}714440$ & $4{.}03\times 10^{-3}$ & $1{.}71\times 10^{-5}$\\ 
 \hline
 $\frac{\pi}{8}$ & $0{.}692914$ & $7{.}40 \times 10^{-4}$ & $1{.}43\times 10^{-5}$ & 
                   $0{.}693017$ & $1{.}30\times 10^{-3}$ & $3{.}04\times 10^{-6}$\\ 
 \hline
 $\frac{\pi}{12}$ & $0{.}678197$ & $1{.}47 \times 10^{-4}$ & $1{.}26\times 10^{-6}$ &  
                    $0{.}678216$ & $2{.}60\times 10^{-4} $& $2{.}67\times 10^{-7}$\\ 
 \hline
\end{tabular}
\end{center}
\caption{The table of the optimal parameters $d^*(c)$ and the corresponding errors for cubic $G^1$ approximants. In the last three columns are the optimal parameters $d(c)$ according to the radial error and the corresponding errors. Note that in \cite{Vavpetic-Zagar-2018-optimal-circle-arcs}, the forms of control points are different, so the optimal parameters $d$ in the paper differ for the factor $\sqrt{1-c^2}$ from the parameters in the table.}
\label{table:G1}
\end{table}
\begin{figure}[h!]
\begin{minipage}{.5\textwidth}
  \centering
 \includegraphics[width=0.9\textwidth]{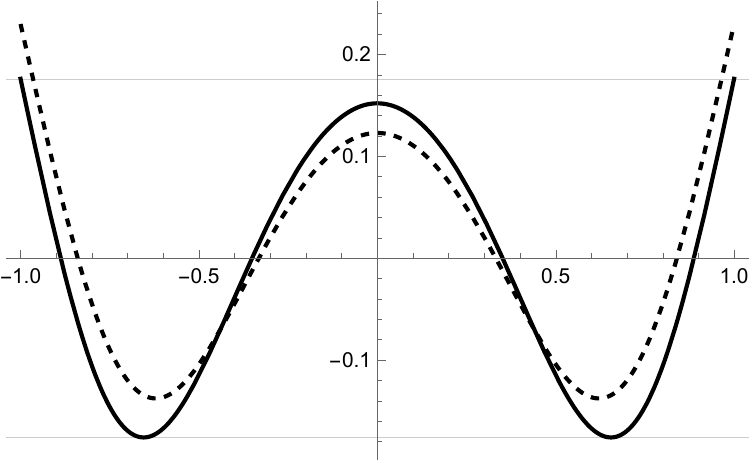}
\end{minipage}
\begin{minipage}{.5\textwidth}
\centering
  \includegraphics[width=0.9\textwidth]{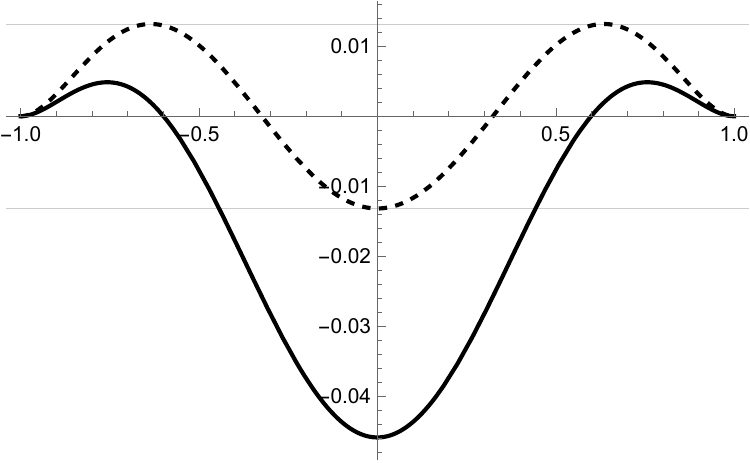}
\end{minipage}
  \caption[]{On the left, there are the graphs of the curvature error of the optimal approximant with respect to the curvature error (solid) and the optimal one with respect to the radial distance (dashed) for the half circle. On the right, there are graphs showing the radial error of the optimal approximant with respect to curvature error (solid) and the optimal one with respect to the radial distance (dashed) again for the half circle.}
  \label{fig:e3compare}
\end{figure}

\section{Quartic $G^2$ approximants}\label{sec:G2_quartic}

As in the previous sections, angle $\varphi\in (0,\tfrac\pi 2]$ is fixed and $c=\cos\varphi\in [0,1)$. Let us first look at the case $c=0$, which is very different from the others. If the boundary control points are $\bfm{b}_0=(0,-1)^T$ and $\bfm{b}_4=(0,1)^T$, $G^2$ condition forces that $\bfm{b}_1=\left(\tfrac{\sqrt{3}}2,-1 \right)^T$ and $\bfm{b}_3=\left(\tfrac{\sqrt{3}}2,1 \right)^T$. We set $\bfm{b}_2=(d,0)^T$ and we get
$$
e_4(t,d)=1-\frac{2 \left(2 t^2 \left(3-t^2\right)+\sqrt{3} d \left(1-t^2\right)^2\right)}{\sqrt{\left(3-3\left(2-d^2\right) t^2+\left(3+4 \sqrt{3} d-6 d^2\right) t^4+\left(4-4 \sqrt{3} d+3 d^2\right)t^6\right)^3}}=:1-\frac{f(t,d)}{\sqrt{g(t,d)^3}}.
$$
Since $e_4(0,d)=1 - \tfrac{2 d}{3}$ the function $e_4(0,\cdot)$ is strictly decreasing. Note that $e_4(1,d) =0$. 

The graphs of the functions $e_4 \left(\cdot,\tfrac 3 2 \right)$ and $e_4 \left(\cdot,\tfrac{8\sqrt{3}}9 \right)$ are on \Cref{fig:bound_for4_c0}. This are the functions with the following properties: $e_4 \left(0,\tfrac 3 2 \right)=0$ and $\tfrac{de_4}{dt}\left(1,\tfrac{8\sqrt{3}}9\right)=0$.
\begin{figure}[h!]
\centering\includegraphics[width=0.5\textwidth]{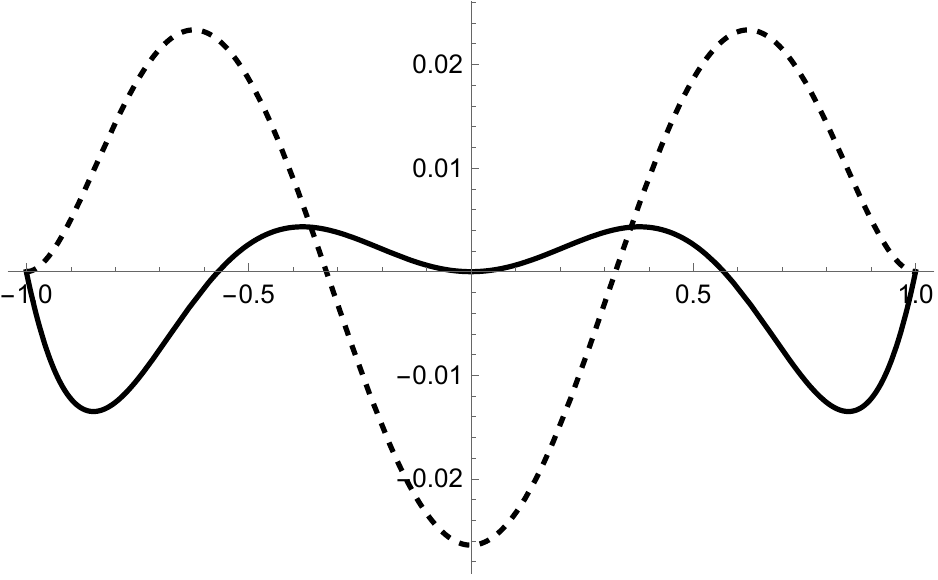}
  \caption[]{The graphs of the functions $e_4(\cdot,\tfrac 3 2)$ and $e_4(\cdot,\tfrac{8\sqrt{3}}9)$; for the second function the graph is dashed.}
  \label{fig:bound_for4_c0}
\end{figure}
We see that $\max |e_4(\cdot,\tfrac 3 2)|=-\min e_4(\cdot,\tfrac 3 2)$ and the minimum is on the interval $[\tfrac 3 5,1]$. Since \[
\tfrac{d e_4}{d d}(t,d)=-\left(\tfrac{d f}{d d}(t,d)g(t,d)-\tfrac 3 2f(t,d)\tfrac{d g}{d d}(t,d)\right)g(t,d)^{-\frac 5 2}=: h(t,d)g(t,d)^{-\frac 5 2}
\]
and
$
 h\left((1-\tau)\frac{3}{5} + \tau,  \frac{8\sqrt{3}}{9} \delta \right)
$
has positive coefficients in the basis $\mathcal{B}_\tau^{10}\times\mathcal{B}_\delta^{2}$, for every $t\in [\tfrac 3 5,1]$ the function $e_4(t,.)$ is strictly increasing for $d\in [0, \frac{8\sqrt{3}}{9}]$. Therefore, every $d<\tfrac 3 2$ induces a worse approximant than the parameter $d=\tfrac 3 2$. For the second function we see that $\max |e_4(\cdot,\tfrac{8\sqrt{3}}9)|=-\min e_4(\cdot,\tfrac{8\sqrt{3}}9)=-e_4(0,\tfrac{8\sqrt{3}}9)$. Since $e_4(0,d)$ is strictly decreasing, every $d>\tfrac{8\sqrt{3}}9$ induces a worse approximant than the parameter $d=\tfrac{8\sqrt{3}}9$. Therefore, we observe that $d^*(0)\in [\tfrac 3 2,\tfrac{8\sqrt{3}}9]$. 

There is a parameter $d\in [\tfrac 3 2,\tfrac{8\sqrt{3}}9]$, such that there exists $t\in[0,1]$ such that $e_4(0,d)=e_4(t,d)$ and $\tfrac{de_4}{dt}(t,d)=0$. The first equation is equivalent to $\tfrac{2d}3= f(t,d)g(t,d)^{- \frac{3}{2}}$ and the second is equivalent to $\tfrac{d f}{d t}(t,d)g(t,d)-\tfrac 3 2f(t,d)\tfrac{d g}{d t}(t,d)=0$. Using Gr\"obner basis for polynomials $\tfrac 1 t((\tfrac{2d}3)^2g(t,d)^3-f(t,d)^2)$ and $\tfrac 1 t(\tfrac{d f}{d t}(t,d)g(t,d)-\tfrac 3 2f(t,d)\tfrac{d g}{d t}(t,d))$ with order $d<t$ we see that the system has only one solution where $d\in[\tfrac 3 2,\tfrac{8\sqrt{3}}9]$. We denote the solution by $d^*(0)\approx 1.5112$. Since $\max|e_4(\cdot,d^*(0))|=-e_4(0,d^*(0))=-e_4(t,d^*(0))$, for some $t\in[\tfrac 3 5,1]$, we see that $d^*(0)$ is the optimal parameter.

Let us now consider the case $c\in (0,1)$. The control points are
\begin{align*}
  \bfm{b}_0&=(c,-\sqrt{1-c^2})^T,
  &\bfm{b}_1=(c,-\sqrt{1-c^2})^T+d\,(1-c^2,c\sqrt{1-c^2})^T,\\
\bfm{b}_2&=\left( \frac{3-4d^2(1-c^2)}{3c},0 \right)^T,
 &\bfm{b}_3=(c,\sqrt{1-c^2})^T+d(1-c^2,-c\sqrt{1-c^2})^T, \\
  \bfm{b}_4&=(c,\sqrt{1-c^2})^T,
\end{align*}
where $d>0$. The corresponding curvature error function is
$$
e_4(t,d)=1-\frac{2 c^2 \left(3 \left(3-4 d^2\right) \left(1-t^2\right)^2+8 c d^3 \left(1+3 t^4\right)-6 c d \left(1-t^2\right) \left(1-(5-4 c d)
   t^2\right)\right)}{\sqrt{\left(\left(1-c^2\right) t^2 \left(4 c d t^2+\left(3-4 d^2\right) \left(1-t^2\right)\right)^2+c^2 \left(2+(1-2 c d) \left(1-3
   t^2\right)\right)^2\right)^3}}=: 1-\frac{f(t,d)}{\sqrt{g(t,d)^3}}.
$$
Since we are considering $G^2$ approximation, the function $e_4(1,\cdot)$ is identically zero. Hence, the function $e_4(\cdot,d)$ has no local extremum at $t=1$, and we need more properties of the function $e_4(\cdot,d)$ as for $e_3(\cdot, d)$ in the $G^1$ cubic case. Numerical experiments have shown that the optimal parameter is near $\frac{3 c^2+\sqrt{6} \sqrt{2+c} (1-c)}{2(c^3+2)}$ which is the solution of the equation $e_4(0,d)=0$. Let us define two rational approximations of the solution as
\begin{align*}
d_1(c):= &\frac{1}{192
   \left(c^3+2\right)} \Bigl( 192 \sqrt{3} -8 \left(16+9 \sqrt{3}\right) c+\left(144+37
   \sqrt{3}\right) c^2 +3 \left(592+232 \sqrt{2}-533 \sqrt{3}\right) c^3  \\
   & -\left(2608+1104 \sqrt{2}-2415\sqrt{3}\right) c^4 + \left(1104+408 \sqrt{2}-973 \sqrt{3}\right) c^5 \Bigr),\\
d_2(c):= &\frac{1}{120
   \left(c^3+2\right) } \Bigl(  120 \sqrt{3} -8 \left(4+9 \sqrt{3}\right) c +\left(144-11
   \sqrt{3}\right) c^2+\left(444+588 \sqrt{2}-735 \sqrt{3}\right) c^3 \\ 
   & -\left(652+996 \sqrt{2}-1191\sqrt{3}\right) c^4 + \left(276+408 \sqrt{2}-493 \sqrt{3}\right) c^5 \Bigr),   
\end{align*}
for which it is not hard to see that 
$$
d_0(c):=\frac{1+\sqrt 2}2 c-\frac{\sqrt 2}2< \frac c 2< d_1(c)<d_2(c)<\frac{3 c^2+\sqrt{6} \sqrt{2+c} (1-c)}{4+2 c^3}<d_3(c):=\frac{1-\sqrt{3}}{2}(c-1)+\frac{1}{2} <\frac{3}{2c}
$$ 
holds for all $c\in (0,1)$.
We will show that the optimal parameter $d^*(c)$ is on the interval $I_c=[d_1(c),d_2(c)]$ for all $c\in(0,1)$. It should be noted that the two bounds are quite precise, as their difference is very small, which can be seen in \Cref{fig:e4meji}. 
\begin{figure}[h!]
  \centering
 \includegraphics[width=0.45\textwidth]{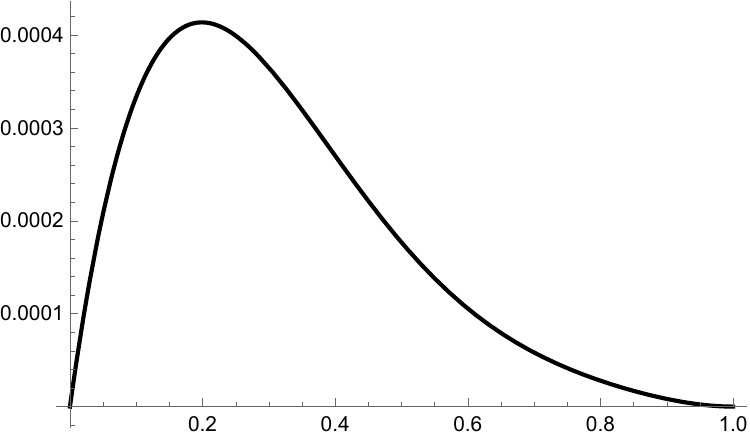}
\caption[]{The difference between the upper bound $d_2$ and the lower bound $d_1$ for the optimal parameter.}
  \label{fig:e4meji}
\end{figure}
Before we prove that $d^*(c)\in I_c$, let us show some properties of the function $e_4$ from which we get the essential properties of functions $e_4(\cdot,d_i(c))$ for further analysis.

\begin{lemma} \label{properties_of_e4}
Let $c\in (0,1)$.
\begin{enumerate}
\item Function $e_4(0,\cdot)$ is strictly increasing on $(\tfrac c 2,\tfrac 3{2c})$ and strictly decreasing on $(-\infty,\tfrac c 2)$.

\item For every $t\in [0{.}45,1)$ the function $e_4(t,\cdot)$ is strictly decreasing on $[d_1(c),d_3(c)]$.

\item For $d\in I_c$ the function $e_4(\cdot,d)$ has exactly 5 local extrema, local minima are at $t=0$ and on $(-1,-0{.}45)\cup (0{.}45,1)$, and $e_4(0,d)<0$, $e_4(0{.}9,d)<0$.
\end{enumerate}
\end{lemma}

\begin{proof}
\begin{enumerate}
\item For every $d<\frac{3}{2c}$ we have
\begin{equation*}
\tfrac{de_4}{dd}(0,d)= \frac{24 (c-2 d)}{c (2 c d-3)^3},
\end{equation*}  
and the statement holds.

\item Since $\tfrac{d e_4}{d d}(t,d)=-\left(\tfrac{d f}{d d}(t,d)g(t,d)-\tfrac 3 2f(t,d)\tfrac{d g}{d d}(t,d)\right)g(t,d)^{- 5/2}=: -h(t,d)g(t,d)^{-5/2}$ and
$$
\frac{(2+c^3)^6}{c^4} h\left((1-\tau)0{.}45+\tau,(1-\delta)d_1(c)+\delta d_3(c)\right)
$$
has positive coefficients in the basis $\mathcal{B}_c^{31}\times\mathcal{B}_\tau^{10}\times\mathcal{B}_\delta^{6}$, the statement holds.

\item The local extremum is at $t=0$. The number of local extrema on the interval $(0,1)$ of the function $e_4(\cdot,d)$ is the same as that of the function $e_4(\sqrt{\cdot},d)$. 
The derivative $\tfrac{d e_4}{dt}(\sqrt t,d)$ is of the form $c^2 h(t,d)g(\sqrt{t},d)^{-5/2}$ for some polynomial $h$. To show that a function $e_4(\sqrt{\cdot},d)$ has two local extrema on $(0,1)$, we have to prove that $h(\cdot,d)$ has two zeros on $(0,1)$. To prove that it suffices to show that the second derivative $\tfrac{d^2 h}{dt^2}(t,d)$ is non-zero for $(t,d)\in(0,1)\times(d_1(c),d_2(c))$. This is true since $(2+c^3)^7\tfrac{d^2 h}{dt^2}(t,(1-\delta)d_1(c)+\delta d_2(c))$ has positive coefficients in the basis $\mathcal{B}_c^{38}\times\mathcal{B}_t^{2}\times\mathcal{B}_\delta^{7}$. Since $e_4(0,d_2(c))<0$ 
and $e_4(0,\cdot)$ is strictly increasing on $I_c$, we have $e_4(0,d)<0$ for all $d\in I_c$. Since $e_4(0{.}45,\cdot)$ is strictly decreasing and $ (2+c^3)^{-12}\left(g(0{.}45,d_2(c))^3-f(0{.}45,d_2(c))^2\right)$ has positive coefficient in the basis $\mathcal{B}_c^{66}$ we have
$$
e_4(0{.}45,d)\ge e_4(0{.}45,d_2(c))=1-\frac{f(0{.}45,d_2(c))}{\sqrt{g(0{.}45,d_2(c))^3}}>0
$$
for all $d\in I_c$. Since $e_4(0{.}9,\cdot)$ is strictly decreasing, $(2+c^3)^{-3}f(0{.}9,d_1(c))$ has positive coefficients in the basis $\mathcal{B}_c^{18}$, and $(2+c^3)^{-12}\left(f(0{.}9,d_2(c))^2-g(0{.}9,d_1(c))^3\right)$ has positive coefficient in the basis $\mathcal{B}_c^{66}$ we have
$$
e_4(0{.}9,d)\le e_4(0{.}9,d_1(c))=1-\frac{f(0{.}9,d_1(c))}{\sqrt{g(0{.}9,d_1(c))^3}}<0
$$
for all $d\in I_c$. Since 
\begin{align*}
\frac{de_4}{dt}(0{.}45,d) &= -\left(\tfrac{d f}{d t}(0{.}45,d)g(0{.}45,d)-\tfrac 3 2f(0{.}45,d)\tfrac{d g}{d t}(0{.}45,d)\right)g(0{.}45,d)^{-\frac 5 2}
 =: \frac{h(0{.}45,d)}{g(0{.}45,d)^{\frac 5 2}}
\end{align*}
and
\[
\frac{(2+c^3)^7}{c^5} h(0{.}45, (1-\delta)d_1(c)+\delta d_2(c)) 
\]
 has all coefficients positive in the basis $\mathcal{B}_c^{35} \times \mathcal{B}_\delta^{7}$, the statement follows.\qedhere
\end{enumerate}  
\end{proof}
From the above properties and the fact that $e_4(1,d)=0$, it follows that the functions $e_4(\cdot,d)$, for $d\in I_c$, have exactly five local extrema on $[-1,1]$. We will see that only for some values of $c$, the error function of the optimal approximant has the minimum at $t=0$ (see \Cref{fig:e4examples}).
\begin{figure}[h!]
\begin{minipage}{.5\textwidth}
  \centering
 \includegraphics[width=0.9\textwidth]{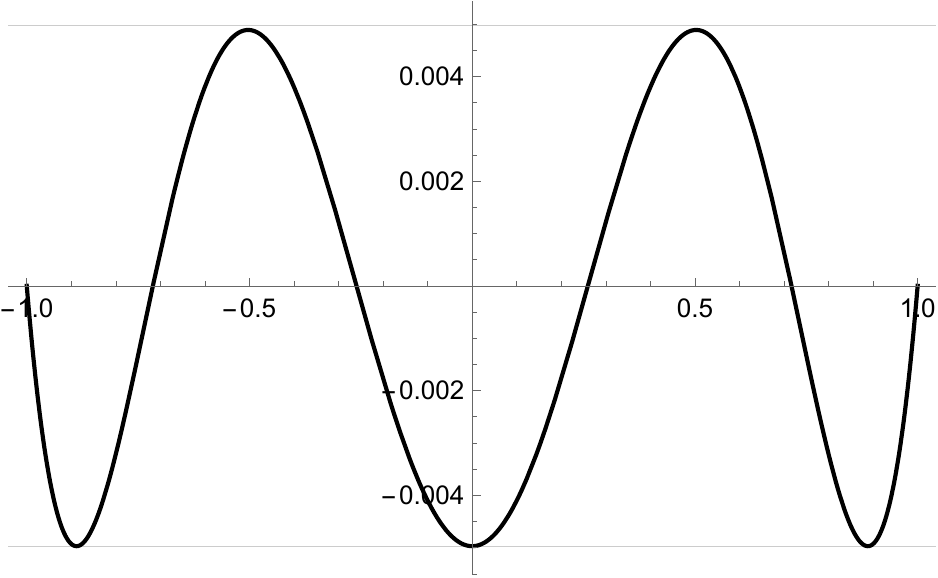}
\end{minipage}
\begin{minipage}{.5\textwidth}
\centering
  \includegraphics[width=0.9\textwidth]{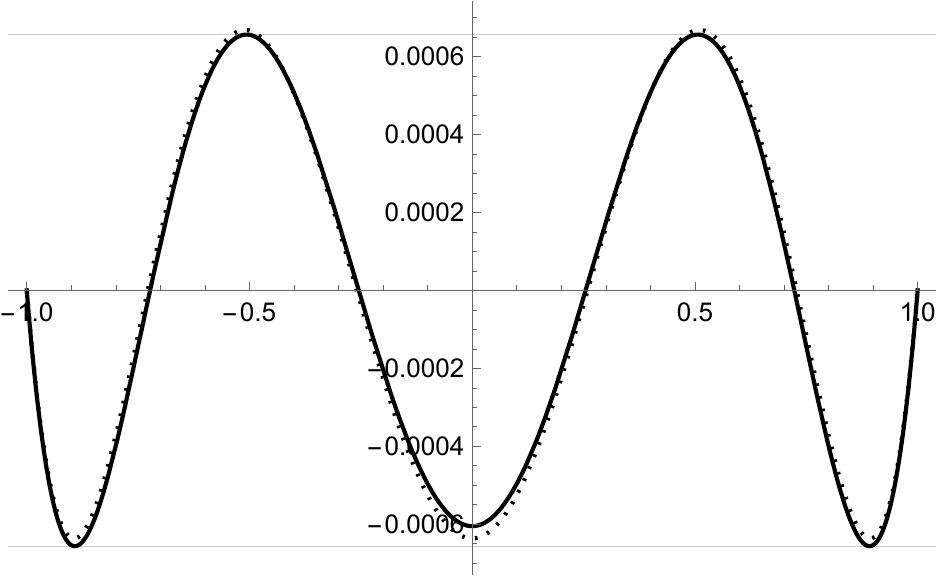}
\end{minipage}
  \caption[]{On the left it is the graph of the error function of the optimal approximant for $c=\tfrac 1{10}$. On the right, the solid graph is the graph of the error function of the optimal approximant for $c=\tfrac 1 2$, and the dotted graph is the graph of the error function where the values of all three minima coincide.}
  \label{fig:e4examples}
\end{figure}
This will be used in the following lemma, which describes some additional properties of the four functions $e_4(\cdot,d_i(c))$ that will be crucial for further analysis.

\begin{lemma} \label{properties_of_e4d}
Let $c\in (0,1)$.
\begin{enumerate}
\item We have $\max_{t\in[-1,1]}|e_4(t,d_1(c))|=-e_4(0,d_1(c))$.
\item There exists $t_0\in [0{.}45,1)$ such that $\max_{t\in[-1,1]}|e_4(t,d_2(c))|=-e_4(t_0,d_2(c))$.
\item We have $|e_4(0,d_0(c))|>\max_{t\in[-1,1]} |e_4(t,d_1(c))|$ and for every $d\in (0,d_0(c))$ it holds $ e_4(0{.}4,d) >\max_{t\in[-1,1]} |e_4(t,d_1(c))|$.

\item We have $ e_4(0,d_3(c)) >\max_{t\in[-1,1]} |e_4(t,d_1(c))|$.
\end{enumerate}
\end{lemma}

\begin{proof}
\begin{enumerate}
\item We have to show that $e_4(0,d_1(c))+e_4(t,d_1(c))\le 0$ and $e_4(0,d_1(c))-e_4(t,d_1(c))\le 0$ for all $t\in [0{.}45,1]$, since $e_4(\cdot,d_1(c))$ has no local extrema on $(0,0{.}45)$ (\Cref{properties_of_e4}).

For the first inequality we have to show that $f(t,d_1(c))\ge (1+e_4(0,d_1(c)))\sqrt{g(t,d_1(c))^3}$. We have $f(t,d_1(c))\ge 0$ for all $t\in[0{.}45,1]$, since $(2+c^3)^3c^{-3}f(\frac{55\tau +45}{100},d_1(c))$ has positive coefficients in the basis $\mathcal{B}_c^{15}\times\mathcal{B}_\tau^4$. Since $e_4(0,d)=1-\tfrac{6-8 d^2}{c (3-2 c d)^2}$, $f(t,d_1(c))^2-(1+e_4(0,d_1(c))^2g(t,d_1(c))^3$ is a rational function and it is positive for $t\in[0{.}45,1]$ since it has the form $c^6(1-c)^3(3-2 c d_1(c))^4(2+c^3)^{-16}h(t,c)$, where $h(\frac{55\tau +45}{100},c)$ has positive coefficients in the basis $\mathcal{B}_c^{81}\times\mathcal{B}_{\tau}^{37}$. 

The second inequality follows since $1-e_4(0,d_1(c))\ge 0$ and $(1-e_4(0,d_1(c))^2g(t,d_1(c))^3-f(t,d_1(c))^2$ is a rational function and it is positive for $t\in [0{.}45,1]$ since it has the form $t^2 c^6(1-c)^3(3-2 c d_1(c))^{-4}(2+c^3)^{-16}h(t,c)$, where $h(\frac{55\tau +45}{100},c)$ has positive coefficients in the basis $\mathcal{B}_c^{75}\times\mathcal{B}_{\tau}^{16}$.

\item By the previous lemma we know that $e_4(\cdot,d_2(c))$ has a local minimum at some $t_0\in[0{.}45,1)$. So it is enough to show that $e_4(0{.}9,d_2(c))<e_4(0,d_2(c))$ and $e_4(0{.}9,d_2(c))+e_4(t,d_2(c))\le 0$ for all $t\in [0,1]$. 

Since $(2+c^3)^3c^{-3}f(0{.}9,d_2(c))$ has positive coefficients in the basis $\mathcal{B}_c^{15}$ we have $f(0{.}9,d_2(c))\ge 0$. Hence, for the first inequality, it is enough to show that $f(0{.}9,d_2(c))^2g(0{.}9,d_2(c))^{-3}>f(0,d_2(c))^2g(0,d_2(c))^{-3}$. This follows since
$$
\frac{(2+c^3)^{16}}{c^{12}(3-2c d_2(c))^2}\left(f(0{.}9,d_2(c))^2g(0,d_2(c))^3-g(0{.}9,d_2(c))^3f(0,d_2(c))^2 \right)
$$
has positive coefficients in the basis $\mathcal{B}_c^{78}$. For the second inequality, 
we have to prove $2\sqrt{g(0{.}9,d_2(c))^3g(t,d_2(c))^3}-f(0{.}9,d_2(c))\sqrt{g(t,d_2(c))^3}-f(t,d_2(c))\sqrt{g(0{.}9,d_2(c))^3}\le 0$. We have shown above that $f(t,d_2(c))\ge 0$ for all $t\in [0{.}45,1]$, so it is enough to show that 
\begin{align*}
4g(0{.}9,d_2(c))^3g(t,d_2(c))^3\le\left(f(0{.}9,d_2(c))\sqrt{g(t,d_2(c))^3}+f(t,d_2(c))\sqrt{g(0{.}9,d_2(c))^3}\right)^2,
\end{align*}
which is equivalent to
\begin{multline*}
l(t,c):=4g(0{.}9,d_2(c))^3g(t,d_2(c))^3-f(0{.}9,d_2(c))^2g(t,d_2(c))^3-f(t,d_2(c))^2g(0{.}9,d_2(c))^3\le \\2f(0{.}9,d_2(c))f(t,d_2(c))\sqrt{g(t,d_2(c))^3g(0{.}9,d_2(c))^3}=:r(t,c).
\end{multline*}
Then $l(t,c)\ge 0$ for all $(c,t)\in [0,1]\times[0{.}45,1]$ since $c^{12}(2+c^3)^{-24}l(\frac{55\tau +45}{100},c)$ has positive coefficients in the basis $\mathcal{B}_c^{120}\times \mathcal{B}_\tau^{18}$, and $r(t,c)^2-l(t,c)^2\ge 0$ for all $(c,t)\in [0,1]\times[0{.}45,1]$ since $c^{24}(2+c^3)^{-48}\left(r(\frac{55\tau +45}{100},c)^2-l(\frac{55\tau +45}{100},c)^2\right)$ has positive coefficients in the basis $\mathcal{B}_c^{264}\times \mathcal{B}_\tau^{67}$.
Hence, the second inequality follows.

\item Due to $e_4(0,d_0(c))<0$ the first inequality is equal to $-1+\tfrac{f(0,d_0(c))}{\sqrt{g(0,d_0(c))^3}}>-1+\tfrac{f(0,d_1(c))}{\sqrt{g(0,d_1(c))^3}}$. The inequality follows from the fact that $\tfrac{(2 + c^3)^3}{c^5(1-c)^3}\left(f(0,d_0(c))\sqrt{g(0,d_1(c))^3}-f(0,d_1(c))\sqrt{g(0,d_0(c))^3} \right)$ is polynomial with positive coefficients in the basis $\mathcal{B}_c^{17}$.

Since $c^{-2}f(0{.}4,d_0(c)\delta)$ is a polynomial with positive coefficients in the basis $\mathcal{B}_c^{4}\times\mathcal{B}_\delta^{3}$ we have
$f(0{.}4,d_0(c)\delta)>0$ for $\delta \in [0,1]$. So the second inequality is equivalent to 
\begin{align*}
e_4(0{.}4,d_0(c)\delta) +e_4(0,d_1(c))=1 +e_4(0,d_1(c))-\frac{f(0{.}4,d_0(c)\delta)}{\sqrt{g(0{.}4,d_0(c)\delta)^3}}&>0,\\
(2+c^3)^{4}(3-2cd_1(c))^4\left((1 +e_4(0,d_1(c))^2g(0{.}4,d_0(c)\delta)^3-f(0{.}4,d_0(c)\delta)^2\right)&>0.
\end{align*}
The inequality holds since the left side of the last inequality is a polynomial with positive coefficients in the basis $\mathcal{B}_c^{42}\times\mathcal{B}_\delta^{12}$,

\item Due to $e_4(0,d_3(c))>0$, we have to show that  $e_4(0,d_3(c))+e_4(0,d_1(c))>0$. Since
$$
\frac{(2+c^2)^2}{c^2}\left(e_4(0,d_3(c))+e_4(0,d_1(c))\right)\sqrt{g(0,d_3(c))^3g(0,d_1(c))^3}
$$
is a polynomial with positive coefficients in the basis $\mathcal{B}_c^{16}$, the inequality follows.
\qedhere
\end{enumerate}
\end{proof}

From the above properties of the functions $e_4(\cdot,d_i(c))$, it is now easy to show that the optimal parameter is on the interval $I_c$.

\begin{lemma} \label{bounds_for_d*}
For every $c \in  (0,1)$ we have $d^*(c)\in I_c$.
\end{lemma}

\begin{proof}
Since $\tfrac{de_4}{dd}(0,d)=\tfrac{24(2d-c)}{c(2cd-3)^3}>0$ for $d>\tfrac 3{2c}$ and $\lim_{d\to\infty} e_4(0,d)=1-2c^{-3}<-1$, we have $e_4(0,d)<-1$ for all $d>\tfrac 3{2c}$. For $d=\tfrac 3{2c}$ the curvature at $t=0$ is not defined, so $d^*(c)<\tfrac 3{2c}$.

For $d_3(c)\le d<\tfrac 3{2c}$, we have $e_4(0,d)\ge e_4(0,d_3(c))>\max_{t\in[-1,1]} |e_4(\cdot,d_1(c))|$. For $d_2(c)<d\le d_3(c)$, we have $e_4(t_0,d)<e_4(t_0,d_2(c)) = - \max_{t\in[-1,1]} |e_4(\cdot,d_2(c))| $ for some $t_0\in [0{.}45,1)$. For $d_0(c)\le d<d_1(c)$, we have $e_4(0,d)<e_4(0,d_1(c))$. For $0<d<d_0(c)$, we have $e_4(0{.}4,d)>|e_4(0,d_1) |=\max_{t\in[-1,1]} |e_4(\cdot,d_1(c))|$. Therefore, $d^*(c)\in I_c$. 
\end{proof}

By lemma \ref{properties_of_e4} for every $d\in I_c$ the function $e_4(\cdot,d)$ has a local minimum at $t=0$ and two local extrema on $(0{.}45,1)$. Let us denote by $t_{M}(d)$ and $t_{m}(d)$ the points on $(0,1)$ at which the function $e_4(\cdot,d)$ has the local maximum and local minimum, respectively.

Since $e_4(0,d_1(c) ) < e_4(t_{m}(d_1(c)),d_1(c) ) $, $e_4(t_{m}(d_2(c)),d_2(c)) < e_4(0,d_2(c)) $, $e_4(0,\cdot)$ is strictly increasing on $I_c$ and $e_4(t_{m}(\cdot),\cdot)$ is strictly decreasing on $I_c$, there exists a unique $d_e(c) \in I_c$ such that $e_4(0,d_e(c)) = e_4(t_{m}(d_e(c)) ,d_e(c))$.

\begin{theorem}\label{mainG2}
Let $c\in (0,1)$.
\begin{enumerate}
\item If $e_4(t_{M}(d_e(c)),d_e(c)) \le |e_4(0,d_e(c))|$, then $d^*(c)=d_e(c)$.

\item If $e_4(t_{M}(d_e(c)),d_e(c)) > |e_4(0,d_e(c))|$, then there exists a unique $d>d_e(c)$ such that \\ 
$e_4(t_{M}(d),d) = |e_4(t_{m}(d),d)|$. Then $d^*(c)=d$.
\end{enumerate}
\end{theorem}

\begin{proof}
\begin{enumerate}
\item Since $e_4(t_{m}(\cdot),\cdot)$ is decreasing on $I_c$ every $d>d_e(c)$ induces a worse approximant and since $e_4(0,\cdot)$ is increasing on $I_c$ every $d<d_e(c)$ induces a worse approximant, hence $d^*(c)= d_e(c)$.

\item  
Since functions $e_4(t_{m}(\cdot),\cdot)$ and $e_4(t_{M}(\cdot),\cdot)$ are both decreasing on $I_c$, we have
$$
e_4(t_{M}(d_e(c)),d_e(c)) > |e_4(t_{m}(d_e(c)) ,d_e(c))|\text{ and }e_4(t_{M}(d_2(c)),d_2(c))\le |e_4(t_{m}(d_2(c)) ,d_2(c))|.
$$ 
Hence there exists a unique $d \in I_c$ such that  $e_4(t_{M}(d),d) = |e_4(t_{m}(d),d)|$. Since $e_4(t_{M}(\cdot),\cdot)$ is decreasing on $I_c$ every $d' <d $ induces a worse approximant and since $e_4(t_{m}(\cdot),\cdot)$ is decreasing on $I_c$ every $d' > d$ induces a worse approximant, hence $d^*(c)= d$.\qedhere
\end{enumerate}
\end{proof}

The numerical solution of equations $e_4(0,d(c))= e_4(t_{m}(d(c)),d(c)) = - e_4(t_{M}(d(c)),d(c))$ is $c = 0.181857$. We can conclude that $d^*(c) = d_e(c)$ for $c < 0.181857$, i.e. for $\varphi > 0{.}44178\pi$, otherwise we have to increase $d$.

In \Cref{table:G2}, numerical values of optimal parameters $d^*$ according to the curvature error and optimal parameters $d^r$ according to the radial error \cite{Vavpetic-Zagar-2018-optimal-circle-arcs} are shown for several values of an inner angle $2\varphi$.

\begin{table}[ht]
\begin{center}
\begin{tabular}{ |c|c|c|c|c|c|c| } 
 \hline
 $\varphi$  & $d^*$ & curvature error & radial error & $d^r$ & curvature
   error & radial error \\ \hline
 $\frac{\pi }{2}$ & $1.511152$ & $7.43473\times 10^{-3}$ & $1.07833\times 10^{-3}$ & $1.513820$ & $9.21346\times 10^{-3}$ & $6.95275\times 10^{-4}$ \\ \hline
 $\frac{\pi }{3}$ & $0.631866$ & $6.89404\times 10^{-4}$ & $3.24026\times 10^{-5}$ & $0.631836$ & $7.7361\times 10^{-4}$ & $2.62103\times 10^{-5}$ \\ \hline
 $\frac{\pi }{4}$ & $0.569351$ & $1.25139\times 10^{-4}$ & $2.99058\times 10^{-6}$ & $0.569344$ & $1.35554\times 10^{-4}$ & $2.59234\times 10^{-6}$ \\ \hline
 $\frac{\pi }{6}$ & $0.529434$ & $1.10948\times 10^{-5}$ & $1.11141\times 10^{-7}$ & $0.529429$ & $1.17661\times 10^{-5}$ & $1.00281\times 10^{-7}$ \\ \hline
 $\frac{\pi }{8}$ & $0.516294$ & $1.98032\times 10^{-6}$ & $1.09529\times 10^{-8}$ & $0.516294$ & $2.08568\times 10^{-6}$ & $1.00093\times 10^{-8}$ \\ \hline
 $\frac{\pi }{12}$ & $0.507161$ & $1.74187\times 10^{-7}$ & $4.22729\times 10^{-10}$ & $0.507161$ & $1.82575\times 10^{-7}$ & $3.89715\times 10^{-10}$ \\
 \hline
\end{tabular}
\end{center}
\caption{The table of the optimal parameters $d^*(c)$ and the corresponding errors for quartic $G^2$ approximants. In the last three columns are the optimal parameters $d^r(c)$ according to the radial error and the corresponding errors. Note that in \cite{Vavpetic-Zagar-2018-optimal-circle-arcs}, the forms of control points are different, so the optimal parameters $d$ in the paper differ for the factor $\sqrt{1-c^2}$ from the parameters in the table. Note that the parameter for the half circle ($\varphi = \frac{\pi}{2}$) has a different meaning than for the other cases (see the definition of the control points). }
\label{table:G2}
\end{table}
\begin{figure}[h!]
\begin{minipage}{.5\textwidth}
  \centering
 \includegraphics[width=0.9\textwidth]{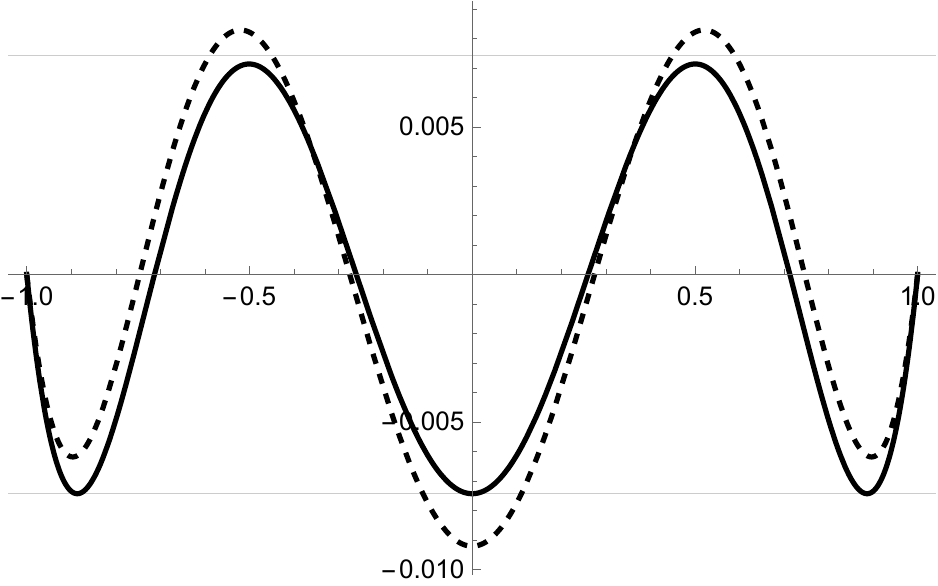}
\end{minipage}
\begin{minipage}{.5\textwidth}
\centering
  \includegraphics[width=0.9\textwidth]{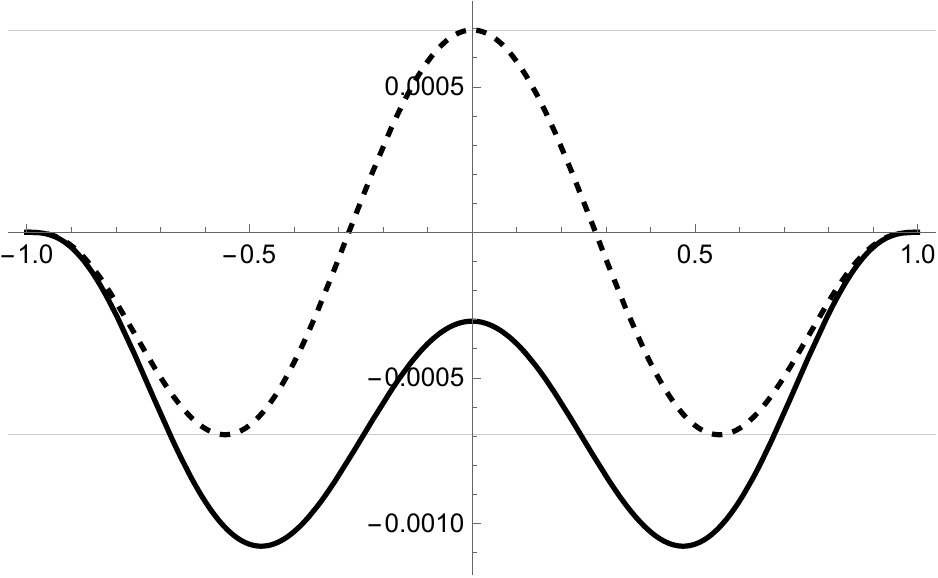}
\end{minipage}
  \caption[]{On the left, there are the graphs of the curvature error of the optimal approximant with respect to the curvature error (solid) and the optimal one with respect to the radial distance (dashed) for the half circle. On the right, there are graphs showing the radial error of the optimal approximant with respect to curvature error (solid) and the optimal one with respect to the radial distance (dashed) again for the half circle.}
  \label{fig:e4compare}
\end{figure}



\section{Conclusion}\label{sec:conclusion}

In this paper, we have confirmed that the cubic $G^1$ parametric approximant given in \cite{Kovac-Zagar-curvature16}, whose curvatures at the middle and at the boundary point coincide, is indeed, under some conditions, the best fit for a circular arc in terms of curvature. This only holds for a circular arc whose inner angle is less than a certain angle. We have shown that all other circular arcs' optimal approximants satisfy a different condition. A similar characterization for the optimal quartic $G^2$ approximant was done. The nature of the curvature function makes the analysis of parametric polynomials of a higher degree too complex. Still, numerical results show that the $G^{n-2}$ approximation by a parametric polynomial of degree $n$ yields a similar characterization of the optimal approximants. Also, the problem of finding the optimal quartic $G^1$ interpolant is a challenging issue, and we do not know the characterization of optimal approximants.

\vskip3mm
{\noindent \sl Acknowledgments.}
The authors are grateful to Emil \v{Z}agar for numerous fruitful discussions, valuable comments and suggestions.
The first author was supported by the Slovenian Research and Innovation Agency program P1-0294. The second author was supported by the Slovenian Research and Innovation Agency program P1-0292 and the grant J1-4031.

\bibliographystyle{elsarticle-harv}
\bibliography{curvature}

\end{document}